\documentclass[reqno, 12pt]{amsart}

\usepackage{mathrsfs}
\usepackage{amscd}
\usepackage{amsmath}
\usepackage{latexsym}
\usepackage{amsfonts}
\usepackage{amssymb}
\usepackage{amsthm}
\usepackage{graphicx}
\usepackage{hyperref}
\usepackage{makecell}
\usepackage{array}
\usepackage{booktabs}
\usepackage{multirow}
\usepackage{color,xcolor}

\newtheorem{theorem}{Theorem}[section]

\newtheorem{lemma}[theorem]{Lemma}

\newtheorem{remark}[theorem]{Remark}

\numberwithin{equation}{section}

\title[Inverse moving source problem]{Inverse moving source problems for parabolic equations}

 \author[Y. Zhao]{Yue Zhao}
\address{School of Mathematics and Statistics, Central China Normal University,
Wuhan 430079, China}
\email{zhaoyueccnu@163.com}


\begin{document}

\begin{abstract}
This paper is concerned with the inverse moving source
problems for parabolic equations. 
Given the temporal function, we prove the uniqueness of the nonlinear inverse problem of determining the orbit function by final data measured in a bounded domain.
On the other hand, given the orbit function we also show that the profile function can be uniquely determined by final data measured in a bounded domain away from 
the domain enclosing the moving orbit.
The proofs adopt the Fourier approach and results from complex analysis. 
\end{abstract}

\maketitle

\section{Introduction}

We consider the following parabolic system
\begin{align}\label{main_eq}
\begin{cases}
\partial_t u(x, t) - \Delta u(x, t) =F(x, t), \quad & (x, t) \in \mathbb R^3\times (0, \infty),\\
u(x, 0) = 0,  \quad & x\in\mathbb R^3.
\end{cases}
\end{align}
The parabolic system \eqref{main_eq} describes the diffusion process caused by the internal source $F(x, t)$ which arises in many scientific and industrial areas.
For instance, a classical diffusion process modeled by \eqref{main_eq} is heat conduction, where $u(x, t)$ is the temperature distribution in the medium
and $F(x, t)$ is the heat source which causes the heat propagation in the medium. The parabolic system is also used to describe many time-dependent phenomena, including particle diffusion and pricing of derivative investment instruments \cite{Triki}.

The inverse source problems for the parabolic system \eqref{main_eq} are to determine the unknown internal source $F(x, t)$ from either boundary or final time measurements of $u(x, t)$. The inverse source problems find a variety of applications in science and engineering since the internal source function is hard to measure in many situations.
For example, the inverse source problems are related to the identification problem of sources of pollution on water surfaces like rivers and lakes \cite{Ok}. For this problem,
one considers a diffusion-convection-reaction equation which can easily be reduced to the parabolic equation by a standard transformation. In such a situation, it is usually dangerous to measure the source directly or make observations close to it since the source could be toxic or radiative.
The inverse source problems also have applications in multi-wave imaging, geophysics \cite{Ammari, DY} and identification of the rate of the outflow in an oil reservoir containing a number of walls \cite{BD}. Recently, the inverse source problem approach to the iron loss determination has been studied in \cite{HHP, KSDV}.

The inverse source problems for parabolic equations have been extensively studied both in theory and practice, see e.g. the monograph \cite{Isakov} and \cite{BD, Cannon, CL-AML, CY, HHP, ZLL}. A typical inverse source problem considered in these literatures is to assume that the source function has the form $F(x, t) = f(x)g(t)$ and determine the spatial function $f(x)$ given the temporal function $g(t)$ \cite{CL-AML, CY}. For such sources, the inverse problem is $linear$.
On the other hand, a common situation in practice is that the source moves with respect to time. In this case, 
an important inverse problem is to determine the unknown moving orbit for the source function. Such kind of inverse source problems is highly $nonlinear$ and is rarely studied. 
So far the uniqueness result is not available for the inverse parabolic moving source problem,
which is the focus of this work.  
We mention a recent work on uniqueness of the inverse moving source problem in electrodynamics \cite{HKLZ}. 
The proof in \cite{HKLZ} depends on the strong Huygens principle for wave propagation in three dimensions and the Fourier transform for the time variable,
which converts the time domain problem into a frequency domain one. However, as the heat equations differs dramatically from the wave equations, new methods 
shall be developed.

Now we formulate the inverse moving source problem. In this setting,
the source function $F(x, t)$ is assumed to be given in the following form:
\[
F(x, t) = f(x - a(t))g(t),
\]
where $f:\mathbb R^3\to\mathbb R$ is the source profile function, $g:\mathbb R^+\to \mathbb R$ the temporal function, and
$a: \mathbb R^+\to\mathbb R^3$ is the orbit function of the moving source. 
Denote the ball centered at the origin with radius $R$ by $B_R$ and its boundary by $\partial B_R$. Assume that the profile function $f(x)$ is compactly supported in $B_{R_1}$ and $|a(t)|< R_2$. Thus, the source moves in a bounded domain, i.e. $\text{supp}_x F(x, t)\subset B_R$ with $R > R_1 + R_2$.
Assume that $f\in L^2(B_R)$, $a\in C^1[0, \infty)$ and $g\in C[0, \infty)$ which is also bounded.  There exists a unique solution $u\in L^2(0, T; H^2(\mathbb R^3))$ to \eqref{main_eq}
for any $T>0$ \cite{Isakov}.

This paper is concerned with the inverse moving source problems of determining the profile function $f(x)$ and the orbit function $a(t)$ from the final data measured in a bounded domain. Specifically, we study the following two inverse problems:

\vskip 5pt

\noindent(i) \textbf{IP1}. Given $a(t)$, determine the unknown profile function $f(x)$ from the final data $u(x, T)$ for $T>0$ and $x\in\Omega$. Here $\Omega$ is a bounded domain which satisfies $\Omega\cap (\mathbb R^3\setminus B_R) \neq \phi$.

\vskip 5pt

\noindent(i) \textbf{IP2}. Given $f(x)$, determine the unknown moving orbit $a(t)$ from the final data $u(x, T)$ for $T>0$ and $x\in B_R$.

\vskip 5pt

\noindent The inverse problem \textbf{IP1} is linear, whereas the \textbf{IP2} is a nonlinear inverse source problem.

In this paper, we use the final data to recover either the source profile function or the orbit function. Given the orbit function, we can recover the profile function
using the final data measured in a bounded domain $\Omega$, which satisfies $\Omega\cap (\mathbb R^3\setminus B_R) \neq \phi$. This choice of the domain 
$\Omega$ enables us to apply the unique continuation.  In particular, we point it out that the domain can be chosen such that its closure does not intersect the orbit of
the moving source, i.e., $\bar{\Omega}\cap B_R = \phi$, which could be of practical interest \cite{EI}.  For instance, if the profile source $f$ is radiative, then the observation
data is preferable to be measured in $\mathbb R^3 \setminus B_R$ as it is dangerous to make observations in $B_R$.
The proof uses the Fourier approach and results from complex analysis
such as the Payley--Wiener and Littlewood theorems \cite{ECT}.
On the other hand, given the profile function, we can recover the moving orbit. In this case, we require the measurement domain to be $B_R$, which includes the orbit of the moving source. Motivated by \cite{HKLZ}, we apply the moment theory to deduce the uniqueness under a priori assumptions on the path of the moving source.
The proofs in this paper carry over to two dimensions.

The rest of the paper is organized as follows. Section \ref{IP1} is devoted to the uniqueness of \textbf{IP1}. In Section \ref{IP2} we prove the  uniqueness of \textbf{IP2}.

\section{Determination of the profile function}\label{IP1}

In this section we consider \textbf{IP1}. We adopt the Fourier approach to prove the uniqueness.

The fundamental solution of the parabolic equation is 
\[
G(x, t) = \frac{1}{(4\pi t)^{3/2}} e^{-\frac{|x|^2}{4t}},
\]
which satisfies
\begin{align*}
\begin{cases}
\partial_t G - \Delta G = 0, \quad & (x, t) \in \mathbb R^3\times (0, \infty),\\
G(x, 0) = \delta(x),  \quad & x\in\mathbb R^3.
\end{cases}
\end{align*}
Then from the Duhamel's principle the solution to the parabolic equation \eqref{main_eq} can be represented by
\begin{align}\label{u}
u(x, t) = G(x, t)\ast (f(x-a(t))g(t)) = \int_0^t\int_{B_R} G(x - y, t- s) f(y-a(s))g(s) {\rm d}y {\rm d}s.
\end{align}

Define the Fourier transform with respect to the spatial variable $x$ as follows
\[
\hat{v}(\xi) = \int_{\mathbb R^3} v(x) e^{{\rm i}x\cdot\xi} {\rm d}x.
\]
Since $\hat{G}(\xi, t) = e^{-t|\xi|^2}$, noting \eqref{u}, taking the Fourier transform of $u$ with respect to the spatial variable $x$ and using the 
Fourier transform of convolution $\widehat{w\ast v} = \hat{w}\hat{v}$ one has that
\begin{align}\label{fu}
\hat{u} (\xi, t) &= \int_0^t \hat{G}(\xi, t-s) \widehat{f(x-a(s))}(\xi) g(s){\rm d}s\notag\\
&=\hat{f}(\xi) \int_0^t e^{-{\rm i} (t - s) |\xi|^2} e^{{\rm i}a(s)\cdot\xi} g(s) {\rm d}s.
\end{align}

Denote 
\begin{align}\label{F}
F(\xi) = \int_0^T e^{-{\rm i} (T - s) |\xi|^2} e^{{\rm i}a(s)\cdot\xi} g(s) {\rm d}t.
\end{align}
The following lemma is useful in the subsequent analysis.
\begin{lemma}\label{est}
Let $\zeta = {\rm i} r \eta$ with $\eta\neq 0$ and $r>1$. There exists $C>0$ such that
\[
|F(\zeta)|\geq C e^{C |\zeta|^2}.
\]
\end{lemma}

\begin{proof}
For $\zeta = {\rm i} r \eta$, noting that the integrand in \eqref{F} is positive and $a$ is bounded, one has
\begin{align*}
|F(\zeta)| = \int_0^T e^{(T-s)r^2|\eta|^2} e^{-ra(s)\cdot\eta} {\rm d}t &\geq \int_{T/3}^{2T/3} e^{(T-s)r^2|\eta|^2} e^{-ra(s)\cdot\eta} {\rm d}t\\
&\geq C e^{C |\zeta|^2},
\end{align*}
which completes the proof.
\end{proof}

The following theorem concerns the uniqueness of \textbf{IP1}. The proof adapts the arguments in that of \cite[Theorem 2.1]{EI} to the current case.
\begin{theorem}
Given the orbit function $a$ and $T>0$, the source profile function can be uniquely determined by $u(x, T)$ for $x\in\Omega$.
Here $\Omega$ is a bounded domain such that $\Omega\cap (\mathbb R^3\setminus B_R)$ is not empty.
\end{theorem}

\begin{proof}
Assume that $u_1$ and $u_2$ are solutions of \eqref{main_eq} corresponding to two profile functions $f_1$ and 
$f_2$, respectively. Let $u = u_1 - u_2$. If $u_1(x, T) = u_2(x, T)$ on $\Omega$, then from the unique continuation of Mizohata \cite{Mizo, SS} one has $u(x, T) = 0$ in
$\mathbb R^3\backslash B_R$. Thus, $u(x, T)$ has compact support. Taking the Fourier transform of $u$ and using \eqref{fu} one has
\begin{align}\label{fu1}
\hat{u}(\xi, T)= \hat{f}(\xi)
\int_0^T e^{-{\rm i} (T - s) |\xi|^2} e^{{\rm i}a(s)\cdot\xi} g(s) {\rm d}t.
\end{align}
Since both $f$ and $u$ have compact support, one has from the Payley--Wiener theorem that $\hat{f}$ and $\hat{u}$ are entire analytic function of order one. 
Thus, they are well-defined for $\xi = \zeta\in\mathbb C^3$, and there exists a constant $C$ such that
\begin{align}\label{est1}
|\hat{f}(\zeta)| \leq Ce^{C|\zeta|}, \quad |\hat{u}(\zeta, T)| \leq Ce^{C|\zeta|}.
\end{align}

Now we prove the uniqueness by contradiction. If $f\neq 0$, then there exists $\xi_0\in\mathbb R^3$ such that $\hat{f}(\xi_0)\neq 0$. The single complex variable entire function
$w(z) = \hat{f} (z\xi_0), z\in\mathbb C$ is of order one and not identically zero. Using the Littlewood theorem from complex analysis (see p. 227 of \cite{ECT}) one has that there exist $C$ and a sequence $l_j\to\infty$ such that $\min_{|z|=l_j}|w(z)|>e^{-Cl_j}$. Using \eqref{fu1} and \eqref{est1} one has that
\[
|F(\zeta)| \leq Ce^{2Cl_j}, \quad \text{for} \,\, |\zeta| = l_j,
\]
which is a contradiction to Lemma \ref{est}. Thus we have $f \equiv 0$ which completes the proof.

\end{proof}

\section{Determination of the orbit function}\label{IP2}

In this section we consider \textbf{IP2}.  The following lemma \cite[Lemma 4.1]{HKLZ} is useful in the proof of the uniqueness.

\begin{lemma}\label{moment}
Let $f_1, f_2, g\in C^1[0, L]$ be functions such that
\[
f^{\prime}_1>0, f^{\prime}_2>0, g>0 \mbox{ on } [0,L];\quad f_1(0)=f_2(0).
\]
In addition, suppose that
\begin{align}\label{eq:f}
\int_{0}^{L} f_1^n( s)g(s){\rm d} s = \int_{0}^{L} f_2^n( s)g(s){\rm d} s
\end{align}
for all integers $n = 0, 1, 2 \cdots$. Then it holds that $f_1 = f_2$ on $[0, L]$.
\end{lemma}

The following theorem is the main result of this paper which proves the uniqueness of \textbf{IP2}.

\begin{theorem}\label{ip2}
Assume that $g\in C(0, T)$, $g(t)>0$ for $t\in (0, T)$ and that $a(0) = 0$ is located at the origin and that each component $a_j, j = 1, 2, 3$ of $a$ satisfies $a_j^\prime\neq 0$ for $t\in (0, T)$.
The orbit function $a$ can be uniquely determined by the final data $u(x, T)$ for $x\in  B_R$.
\end{theorem}

\begin{proof}
Since $u_1(x, T) = u_2(x, T)$ for $x\in B_R$, from unique continuation in  \cite{Mizo, SS} one has that $u_1(x, T) = u_2(x, T)$ for all $x\in\mathbb R^3$, which gives 
$\hat{u_1}(\xi, T) = \hat{u_2}(\xi, T)$.
Using \eqref{fu} one has that
\begin{align}\label{iden}
\hat{f}(\xi)\int_0^T e^{-{\rm i} (T - s) |\xi|^2} e^{{\rm i}a(s)\cdot\xi} g(s) {\rm d}t = \hat{f}(\xi)
\int_0^T e^{-{\rm i} (T - s) |\xi|^2} e^{{\rm i}b(s)\cdot\xi} g(s) {\rm d}t.
\end{align}
Let $e_1 = (1, 0, 0)$. Since $f \not\equiv 0$, there exists a sequence $\{\tau_j\}_{j=1}^\infty$ such that $\lim_{j\to\infty}\tau_j = 0$ and for each $\tau_j$
one has $\hat{f}(\tau_j e_1)\neq 0$. Letting $\xi = \tau_j e_1$ in \eqref{iden} one has
\begin{align}\label{iden1}
\int_0^T e^{-{\rm i} (T - s) \tau_j^2} e^{{\rm i}\tau a_1(s)} g(s) {\rm d}t = 
\int_0^T e^{-{\rm i} (T - s) \tau_j^2} e^{{\rm i}\tau b_1(s)} g(s) {\rm d}t.
\end{align}
Expanding $e^{-{\rm i} (T - s) \tau^2} e^{{\rm i}\tau a_1(s)}$ and $e^{-{\rm i} (T - s) \tau^2} e^{{\rm i}\tau b_1(s)}$ into Taylor series at $\tau = 0$
we write \eqref{iden1} as
\begin{align}\label{iden2}
\sum_{n = 1}^\infty \frac{\alpha_n}{n!}\tau_j^n = \sum_{n = 1}^\infty \frac{\beta_n}{n!}\tau_j^n
\end{align}
where 
\[
\alpha_n = \int_0^T a^n_1(s) g(s){\rm d}s, \quad \alpha_n = \int_0^T b^n_1(s) g(s){\rm d}s, \quad n = 1, 2, \cdots.
\]
Since \eqref{iden2} holds for $\tau_j$ with $\lim_{j\to\infty}\tau_j= 0$, we have $\alpha_n = \beta_n$ which gives
\[
\int_0^T a^n_1(s) g(s){\rm d}s = \int_0^T b^n_1(s) g(s){\rm d}s.
\]
Then applying Lemma \ref{moment} we have $a_1(t) = b_1(t)$ for $t\in [0, T]$. Similarly we can prove $a_i(t) = b_i(t), i = 2, 3$ 
for $t\in [0, T]$ by choosing $\xi = \tau_j  e_2$ and $\xi = \tau_j  e_3$ where $e_2 =  (0,1 ,0)$ and $e_3 =  (0, 0, 1)$, respectively, and repeating the above arguments.

\end{proof}

\begin{remark}
In \cite[Theorem 4.2]{HKLZ} for wave equations, the speed of the orbit is required to satisfy $|a^\prime_j|<1, j = 1, 2, 3$. In other words, to guarantee the uniqueness of the orbit  the speed of the moving source can not be supersonic. On the other hand, for the heat equation, since the speed of the propagation of the heat can be considered as infinity, the smallness assumption on the speed of the orbit in \cite{HKLZ} can
be dropped in Theorem \ref{ip2}.  
\end{remark}

%

\section{Conclusion}

For the inverse moving source problems of the parabolic equation, we show that given the temporal function, the moving orbit can be uniquely determined from the final data measured in a bounded domain. The profile function can also be uniquely identified from the final data provided that the moving orbit is known. The proof utilizes the Fourier approach and results from complex analysis. A more challenging continuation of this work is to study the stability estimates of the inverse problems.  
Another interesting direction is to develop numerical methods to identify the orbit. It is also important to consider domains with boundaries where the diffusion occurs, e.g., 
a half-space or plane or bounded domains. In these cases, the Green's function is no longer the one in free space or may not even have an explicit representation,   
and the present method may not be directly applicable. We hope to report the progress on these problems elsewhere.

\end{document}